\documentclass{article}

\usepackage{arxiv}

\usepackage[utf8]{inputenc} 
\usepackage[T1]{fontenc}    
\usepackage{hyperref}       
\usepackage{url}            
\usepackage{booktabs}       
\usepackage{amsfonts}       
\usepackage{nicefrac}       
\usepackage{microtype}      
\usepackage{lipsum}
\usepackage{amsmath,amsthm,amssymb}
\usepackage[all]{xy}
\usepackage{graphicx}
\usepackage{mathrsfs}
\usepackage{ upgreek}
\theoremstyle{plain}
\theoremstyle{remark}

\newtheorem*{convention*}{Convention}
\theoremstyle{plain}
\newtheorem{theorem}{Theorem}[section]

\theoremstyle{definition}

\newtheorem{example}[theorem]{Example}

\newcommand{\bbR}{{\mathbb R}}

\title{A  Mathematical  Model for Whorl Fingerprint}

\author{Ibrahim Jawarneh\\
  Department of Mathematics\\
 Al-Hussein Bin Talal University\\
Ma'an, P.O. Box (20), 71111, Jordan\\
  \texttt{ibrahim.a.jawarneh@ahu.edu.jo} \\
   \And
  Nesreen Alsharman\\
  Computer Science\\
The World Islamic Sciences and Education University\\
 Amman, Jordan\\
  \texttt{nesreen.alsharman@wise.edu.jo} \\
}

\begin{document}
\maketitle

\begin{abstract}
In this paper, different classes of the whorl fingerprint are discussed.  A general dynamical system with a parameter $\theta$ is created using differential equations to simulate these classes by varying the value of $\theta$. The global dynamics is studied, and  the existence and stability of equilibria are analyzed.
The Maple is used to visualize fingerprint’s orientation image as a smooth deformation of the phase portrait of a planar dynamical system.
\end{abstract}

\keywords{Whorl fingerprint \and Concentric whorl \and Spiral whorl \and Composite whorl with "S" core \and Simulations of the whorl fingerprint.}

\section{Introduction} \label{FP introduction}
\hspace{\parindent} 
Fingerprints are a set of raised lines that form unique patterns on the pads of the fingers and thumbs. Everyone leaves parts or entire fingerprints on many things through our daily activities by touching cups, doors, books, etc., so studying fingerprints is important in security especially no two people have been found to have the same fingerprints. One of the early studies about fingerprints appeared in 1892 by Sir Francis Galton in his book, finger prints \cite{Galton1892}. There are three general types of fingerprints; loop, whorl, and arch.  The whorl type occurs in about 25–35\% of all fingerprints, see \cite{Arent2019,Galton1892,Henry1990}. 
\par 
\indent The whorl patterns display in the way that the ridges in the center tend to show a circular orientation with a core to whorl and two  deltas in the right and left sides. The focus of this paper is on three basic categories of whorl fingerprint which are concentric whorl, spiral whorl, and composite whorl with "S" core:
\begin{itemize}
\item Concentric whorl, this pattern represents the most basic form of a whorl in which the core is circular or elliptical in the center of the fingerprint,  see   picture (a) in figure \ref{concentric, spiral and Composite fingerprint}.
\item Spiral whorl,  the ridges flow are in winding way in the center making a spiral core, see picture (b) in figure \ref{concentric, spiral and Composite fingerprint}.
\item Composite whorl with "S" core, this pattern twists its ridges in the way that forms a core in  "S" shape in the center, see picture (c) in  figure \ref{concentric, spiral and Composite fingerprint}.
\end{itemize}
\begin{figure}[ht]
\centering
\begin{minipage}[c]{0.33\linewidth}
\centering
\includegraphics[width=2in]{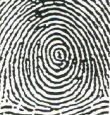}\\ \small (a)
\end{minipage}%
\begin{minipage}[c]{0.33\linewidth}
\centering
\includegraphics[width=2in]{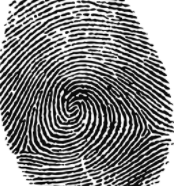}\\ \small (b)
\end{minipage}%
\centering
\begin{minipage}[c]{0.31\linewidth}
\centering
\includegraphics[width=2in]{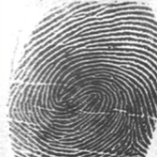}\\ \small (c)
\end{minipage}%
\caption{(a) Concentric whorl,  (b) spiral whorl  and (c) composite whorl}
\label{concentric, spiral and Composite fingerprint}
\end{figure}
\par
\indent Few studies talked about modeling of the whorl fingerprint specially using the phase portraits of a system of differential equations. The idea of phase portraits in texture modelling can be seen in \cite{Rao1990} where a characterizing oriented patterns was proposed using the qualitative differential equation theory to analyze real texture images, but no fingerprint image seems to have been considered. In the thesis by Ford \cite{Ford1994}, the complex flows were divided into simpler components which are modeled by linear phase portraits and then combined to obtain a model for the entire flow field, this idea was applied to fingerprints in \cite{Li2006}. 
\par 
\indent Fingerprint can be captured as graphical ridge and valley patterns, so the global representation of the above categories of fingerprint’s flow-like patterns as a smooth deformation of the phase portrait of a system of differential equations is considered in this paper. Using differential equations in the formularization of the general dynamical system that describes concentric, spiral, and composite whorl fingerprint requires understanding the behaviour of the ridges and interpreting the deltas and cores that appear in these classes, and  how the singular points that represent the core to whorl and deltas in the patterns of phase portraits look like in the considered model.
\par 
\indent The system of differential equations basically has two types of singular points; first type is nondegenerate singular point (The Jacobian matrix of  the system has no zero eigenvalues), the second type is degenerate singular point (The Jacobian matrix has at least one zero eigenvalue). For the core to whorl can be represented as center or spiral singular point in the system of differential equations, but the delta can be interpreted as a singular point with three heperbolic sectors, and there is no such singular point exist. So the closer singular point for the delta is the cusp with cutting appropriate edge of it.   To know more information about the cusp, we present the following theorem, including determining some types of degenerate  equilibrium points of the planar system that are found in \cite{Perko2001}. Assume that the origin is an isolated critical point of the planar system 
  \begin{equation} \label{planar system}
  \begin{aligned} 
    \dot{x} =P(x,y),
    \\
     \dot{y} = Q(x,y).
  \end{aligned} 
  \end{equation}
where $P(x,y)$ and $Q(x,y)$ are analytic in some neighborhood of the origin, and consider the case when the  matrix $A = Df(0)$ has two zero eigenvalues i.e., $det(A) = 0$, $tr (A) = 0$, but $A \neq 0$, in this case  the system \eqref{planar system} can be put in the normal form
\begin{equation}\label{normal form}
 \begin{aligned} 
  \dot{x} &= y,
\\
\dot {y}  &= a_kx^k [1 + h(x)] + b_nx^ny[1 + g(x)] + y^2R(x, y).
\end{aligned}   
\end{equation}
where $h(x)$, $g(x)$ and $R(x, y)$ are analytic in a neighborhood of the origin, $h(0) = g(0) = 0, k \geq 2, a_k \neq 0$ and $n \geq 1$.
\begin{theorem}[Theorem 2 and theorem 3, page 151
\cite{Perko2001}]\label{theorem 2 and theorem 3}.
\begin{enumerate}
    \item Let $k=2m+1$ with $m\geq 1$ in \eqref{normal form} and let $\lambda = b_n^2 + 4(m+1)a_k$. Then if $a_k > 0$, the origin is a (topological) saddle. If $a_k < 0$, the origin  is
    \begin{itemize}
        \item a focus or a center if $b_n =0$ and also if $b_n \neq 0$ and $n > m$ or if $n = m$ and $\lambda < 0$,
        \item a node if $b_n \neq 0$, n is an even number and $n < m$ and also if $b_n \neq 0$, $n$ is an even number, $n = m$ and $\lambda \geq 0$, 
        \item  a critical point with an elliptic domain if $b_n \neq 0$, $n$ is an odd number and $n < m$ and also if $b_n \neq 0$, $n$ is an odd number, $n = m$ and $\lambda \geq 0$.
    \end{itemize}
    \item  Let $k = 2m$ with $m \geq 1$ in \eqref{normal form}. Then the origin is
    \begin{itemize}
        \item a cusp if $b_n = 0$ and also if $b_n \neq 0$ and $n \geq m$,
        \item a saddle-node if $b_n \neq 0$ and $n < m$.
    \end{itemize}
\end{enumerate}
\end{theorem}
 It is clear now from the above theorem that if  the Jacobian matrix $Df(x_0)$ has two zero eigenvalues, then the critical point $x_0$ is either a focus, a center, a node, a (topological) saddle, a saddle-node, a cusp, or a critical point with an elliptic domain.
 \newline
In \cite{zinoun2018fingerprint}, Zinoun used the Taylor polynomials and the normal forms then applied the theorem \eqref{theorem 2 and theorem 3} to formulate some systems of the classes of the whorl fingerprint such as the concentric whorl
\begin{equation}\label{FP-whorl-sys 1st order}
 \begin{aligned} 
  \dot{x} &= y,
\\
\dot {y}  &=  -x(x^2 -1)^2.
\end{aligned}   
\end{equation}
and the spiral whorl fingerprint 
\begin{equation}\label{FP-whorl-sys 1st order}
 \begin{aligned} 
  \dot{x} &= y,
\\
\dot {y}  &= (y -x/2)(x^2 -1)^2.
\end{aligned}   
\end{equation}
In this research, the spiral whorl fingerprint is developed, a new model for composite whorl with "S" core is suggested and  all above categories of the whorl fingerprint are generalized in a dynamical system with a parameter $\theta$ as follows
\begin{equation}\label{FP-whorl-sys}
 \begin{aligned} 
\ddot{x} -({\theta}\dot{x} -x)(x^2 -1)^2 =0, \quad  \theta \in \bbR.
\end{aligned}   
\end{equation}
Equation \eqref{FP-whorl-sys} can be written  as a first order system
\begin{equation}\label{FP-whorl-sys 1st order}
 \begin{aligned} 
  \dot{x} &= y,
\\
\dot {y}  &=   -x(x^2 -1)^2 + {\theta}y(x^2 -1)^2 , \quad  \theta \in \bbR.
\end{aligned}   
\end{equation}
\par
This paper is organized as follows:  In the next section, we study the stability of the equilibria of the system \eqref{FP-whorl-sys 1st order}.  In section \ref{FP-whorl Discussion}  
an interesting simulations are shown for the three categories of the whorl fingerprint. A brief results are summarized in section \ref{FP-whorl conculusion}.
\section{Steady States and  Their Stability} \label{FP-whorl staeady state}
\hspace{\parindent}
To study the stability of the system $\eqref{FP-whorl-sys 1st order}$, we find the equilibria first, which are the solutions of the following equations:
\begin{align} 
    0  & = y,  \label{whorl FP 1st order 1st eqn}
    \\
    0  & =  -x(x^2 -1)^2 + {\theta}y(x^2 -1)^2. \label{whorl FP 1st order 2nd eqn}
\end{align}  
and are given by $E_0 =(0,0)$, $E_1 (1,0)$,and $E_2 = (-1,0)$ which are called equilibrium points or singular points. The Jacobian matrix of $\eqref{FP-whorl-sys 1st order}$ takes the form
\begin{equation} \label{Jacobian Arch}
J =
  \begin{bmatrix}
  0 & 1 \\
-(x^2 -1)^2 + 4({\theta}y -x)(x^2 -1)x & \theta(x^2 -1)^2
\end{bmatrix}.
\end{equation}
The Jacobian matrix evaluated at the equilibrium point $E_0 = (0,0)$ is
\begin{equation} \label{JacobianE_0 whorl}
J(E_0) =
  \begin{bmatrix}
  0 &  1\\
-1& \theta
\end{bmatrix}.
\end{equation}
We summarize the stability of $E_0$ in the following theorem.
\begin{theorem} \label{stability E_0 whorl}
 The equilibrium point $E_0 = (0,0)$ is  stable if $\theta < 0$, unstable if $\theta > 0$, and center if $\theta = 0$.
\end{theorem}
\begin{proof}
From the Jacobian matrix \eqref{JacobianE_0 whorl}, the eigenvalues of $J(E_0)$ are
\begin{equation} \label{eigenvalues of E_0 whorl}
 \lambda_{1,2} = \frac{\theta \pm \sqrt{\theta^2 - 4}}{2}. 
\end{equation}
Notice that the eigenvalues of the equilibrium point $E_0 = (0,0)$ depend only on the parameter $\theta$. When $ \theta < 0$, the real parts of the eigenvalues are negative, so $E_0$ is stable. Similarly, when  $\theta  > 0$, the real parts of the eigenvalues are positive, and $E_0$ is unstable. Finally, when $\theta  = 0$, we get  $ \lambda_{1,2} = \pm i$, and so $E_0$ is center. 
\end{proof}
The equilibrium points $E_1 = (-1,0)$ and $E_2 = (1,0)$ have the same  Jacobian matrix 
\begin{equation} \label{JacobianE_{1,2} whorl}
J(E_1) = J(E_2) =
  \begin{bmatrix}
  0 &  1\\
0& 0
\end{bmatrix}.
\end{equation}
We summarize the stability of $E_{1}=(-1,0)$ and $E_2 =(1,0)$ in the following theorem
\begin{theorem}
 The equilibrium points $E_{1}=(-1,0)$ and $E_2 =(1,0)$ are cusps.
\end{theorem}
\begin{proof}
The eigenvalues in this case are $\lambda_{1,2}$ = 0 which means degenerate  equilibrium points, i.e., $det(J) = 0$, $tr(J) = 0$, but $J \neq 0$. In this case it is shown in the introduction that the system can be put in the normal form to which functions can be reduced in a neighbourhood of the degenerate critical points $E_{1}=(-1,0)$ and $E_2 =(1,0)$ as the following
\begin{equation}\label{FP-whorl-sys 1st order normal form cusps general}
 \begin{aligned} 
  \dot{x} &= y,
\\
\dot {y}  &=   \pm \alpha(x \pm 1)^2 + o((x \pm 1)^3).
\end{aligned}   
\end{equation}
Apply theorem \eqref{theorem 2 and theorem 3}, $k =2, b_n =0, h(x) =0$ and  $\alpha > 0$ which gives that both $E_1 = (-1,0)$ and $E_2 = (1,0)$ are cusps.
\end{proof}
\section{Simulations and Numerical Results of the Whorl Fingerprint} \label{FP-whorl Discussion}
\hspace{\parindent}
In this section, we display the phase portrait of the system \eqref{FP-whorl-sys 1st order} using particular values of the parameter $\theta$ and  match them with the images in the  above categories of the whorl fingerprint. We get similar shape to the  concentric whorl at the value  $\theta =0$, around and closed to $\theta =0$, we get shapes look like to the spiral whorl and when we increase the value of the parameter $\theta$ to be around  one, we get phase portrait closed to the composite whorl with "S" core. For more details let us go over these cases using Maple software. 
\subsection{Concentric whorl class }
\hspace{\parindent} The basic features of the concentric whorl are the existence of the circular or elliptical ridges in the middle which are represented by the center in the phase portrait, the two deltas which are represented by the two cusps in the phase portrait  in the right and left sides, and  we can draw  two  connections between the deltas through the ridges  which are represented by separatrices between the  cusps from above and below half planes in the phase portrait. Let us put $\theta =0$ in the system \eqref{FP-whorl-sys 1st order}, we get the following system
\begin{example}\label{Example Concentric Whorl}
\begin{equation}\label{theta = 0.001}
 \begin{aligned} 
  \dot{x} &= y,
\\
\dot {y}  &=   -x(x^2 -1)^2.
\end{aligned}   
\end{equation}
\end{example}
Figure \ref{Concentric whorl fingerprint} shows the phase portrait of the example \ref{Example Concentric Whorl}  using Maple and the image that represents the concentric whorl. It easy to compare and matching center with circular ridges in the middle , the cusps with the deltas,   and the separatrices  with the connections  in both pictures.
\begin{figure}[ht]
\centering
\begin{minipage}[c]{0.5\linewidth}
\centering
\includegraphics[width= 3.1 in]{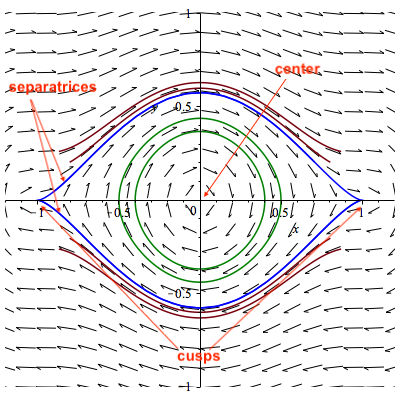}\\ \small (a)
\end{minipage}%
\begin{minipage}[c]{0.5\linewidth}
\centering
\includegraphics[width= 2.8 in]{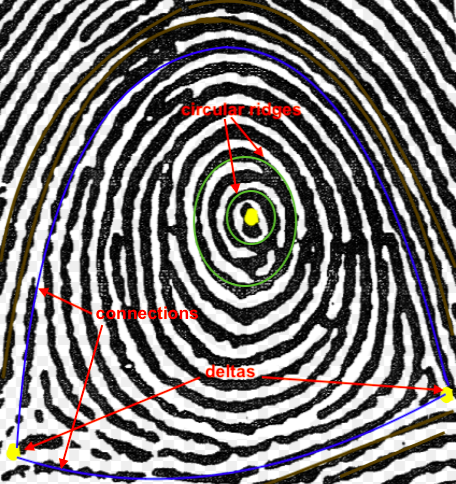}\\ \small (b)
\end{minipage}%
\caption{Simulation of the concentric whorl fingerprint: (a) phase portrait of example  \ref{Example Concentric Whorl} and (b) image of the concentric whorl fingerprint.}
\label{Concentric whorl fingerprint}
\end{figure}
\subsection{Spiral whorl class}
\hspace{\parindent}
In this type, we go over two kinds of the spiral whorl:
\begin{itemize}
    \item UR-LL spiral whorl in which there exist an upper right connection (UR connection) between the spiral core and  the right delta and a lower left connection (LL connection) between the spiral core and the left delta, see picture (b) in figure \ref{UR-LL spiral whorl}.
    \item LR-UL spiral whorl in which there exist a lower right connection (LR connection) between the spiral core and  the right delta and an upper left connection (UL connection) between the spiral core and the left delta, see picture (b) in figure \ref{LR-UL spiral whorl}.
\end{itemize}
To get the first kind of spiral, substitute  $\theta = 0.2$ in the system \eqref{FP-whorl-sys 1st order}, we get the following system
\begin{example}\label{Example UR-LL spiral whorl}
\begin{equation}\label{theta = 0.2}
 \begin{aligned} 
  \dot{x} &= y,
\\
\dot {y}  &=    -x(x^2 -1)^2 + 0.2y(x^2 -1)^2.
\end{aligned}   
\end{equation}
\end{example}
The phase portrait of the example \ref{Example UR-LL spiral whorl} is shown in figure \ref{UR-LL spiral whorl}. We can see the similarity between the phase portrait and  the image that represents  UR-LL spiral whorl  by comparing the focus with the spiral core, the cusps with deltas,   the upper right orbit (UR orbit) between the focus and the right cusp  with the upper right connection (UR connection), and the lower left orbit (LL orbit) between the focus and the left cusp  with the lower left connection (LL connection).
\begin{figure}[ht]
\centering
\begin{minipage}[c]{0.5\linewidth}
\centering
\includegraphics[width= 3.2 in]{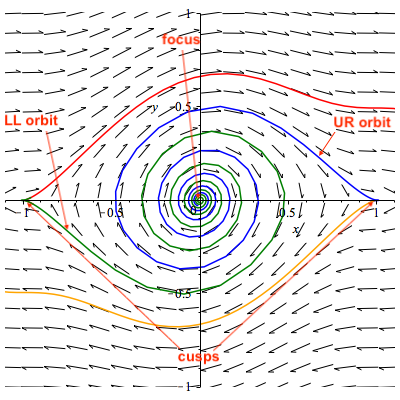}\\ \small(a)
\end{minipage}%
\begin{minipage}[c]{0.5\linewidth}
\centering
\includegraphics[width= 2.2 in]{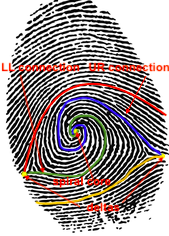}\\ \small(b)
\end{minipage}%
\caption{Simulation of the UR-LL spiral whorl fingerprint: (a) phase portrait of example  \ref{Example UR-LL spiral whorl}  and (b) image of the UR-LL spiral whorl fingerprint.}
\label{UR-LL spiral whorl}
\end{figure}
\par 
\indent To get the kind LR-UL spiral whorl, we should think how to switch the connections in the figure \ref{UR-LL spiral whorl} up side down, this can be happened if we use  negative value of $\theta$, we use $\theta = -0.2$  in the system \eqref{FP-whorl-sys 1st order} which is explained in example \ref{Example LR-UL spiral whorl}.  
\begin{example}\label{Example LR-UL spiral whorl}
\begin{equation}\label{theta = -0.2}
 \begin{aligned} 
  \dot{x} &= y,
\\
\dot {y}  &=    -x(x^2 -1)^2 - 0.2y(x^2 -1)^2.
\end{aligned}   
\end{equation}
\end{example}
The phase portrait of example \ref{Example LR-UL spiral whorl}  is illustrated in figure \ref{LR-UL spiral whorl} which achieves our target. 
\begin{figure}[ht]
\centering
\begin{minipage}[c]{0.5\linewidth}
\centering
\includegraphics[width= 2.7 in]{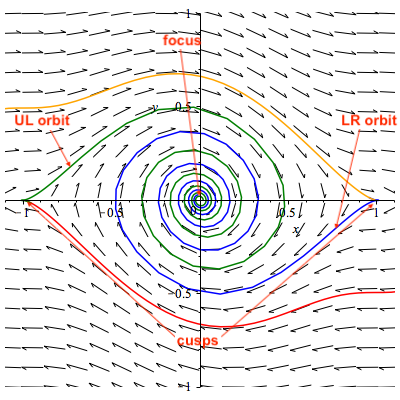}\\ \small(a)
\end{minipage}%
\begin{minipage}[c]{0.5\linewidth}
\centering
\includegraphics[width= 2 in]{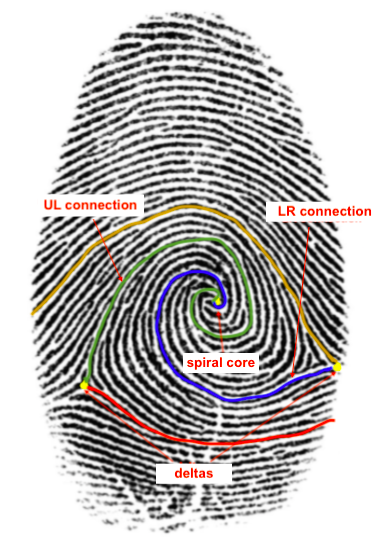}\\ \small(b)
\end{minipage}%
\caption{Simulation of the LR-UL spiral whorl fingerprint: (a) phase portrait of example  \ref{Example LR-UL spiral whorl}  and (b) image of the RU-LL spiral whorl fingerprint.}
\label{LR-UL spiral whorl}
\end{figure}
\subsection{Composite whorl with "S" core class}
\hspace{\parindent} The composite whorl with "S" core  is characterized with a center looks like an "S", and two cusps in the two sides. If $\theta$ in the system \eqref{FP-whorl-sys 1st order} grows up  around one, the flow is twisted enough to create "S" in the middle with keeping the cusps in both sides, for example consider system \eqref{FP-whorl-sys 1st order} with  $\theta = 0.9$ we get the following system
\begin{example}\label{Example Composite whorl with S core}
\begin{equation}\label{theta = 0.9}
\begin{aligned} 
\dot{x} &= y,
\\
\dot {y}  &=    -x(x^2 -1)^2 + 0.9y(x^2 -1)^2.
\end{aligned}   
\end{equation}
\end{example}
Notice the flow in the phase portrait  of the example \ref{Example Composite whorl with S core}  and the friction ridges of the composite whorl with "S" are almost similar. 
\begin{figure}[ht]
\centering
\begin{minipage}[c]{0.5\linewidth}
\centering
\includegraphics[width= 2.8 in]{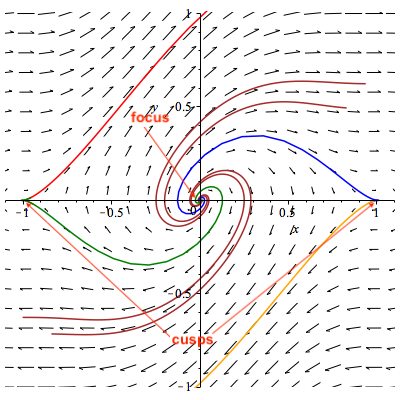}\\ \small(a)
\end{minipage}%
\begin{minipage}[c]{0.5\linewidth}
\centering
\includegraphics[width= 2.4 in]{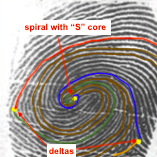}\\ \small(b)
\end{minipage}%
\caption{Simulation of the composite whorl with S core fingerprint: (a) phase portrait of example  \eqref{Example Composite whorl with S core}  and (b) image of the composite whorl with "S" core.}
\label{image of the composite whorl with S core}
\end{figure}
\section{Conclusion} \label{FP-whorl conculusion}
\hspace{\parindent}
The dynamical system \eqref{FP-whorl-sys 1st order} with the parameter  $\theta$ is a good source to generate different categories of the whorl fingerprint. We have noticed that the shape of the flow of this dynamical system at a particular value of the parameter $\theta$ and  the shape of the ridges in the corresponding  image of a category of the whorl fingerprint are almost identical to each other. The flexibility of the system \eqref{FP-whorl-sys 1st order} enable us to simulate a class of whorl depending on how much this class is twisted in either  a clockwise direction or counterclockwise direction. 
\bibliographystyle{unsrt}  


\end{document}